\documentclass[10pt, reqno]{amsart}

\usepackage{amsmath,amssymb,amsthm,amsfonts,verbatim,color,soul,enumerate}
\usepackage{microtype}
\usepackage[all,2cell]{xy}
\usepackage{mathtools}
\usepackage{graphicx}
\usepackage{hyperref}
\usepackage{mathrsfs,pinlabel}

\CompileMatrices

\usepackage[top=1.3in,bottom=1.9in,left=1.4in,right=1.4in]{geometry}

\theoremstyle{plain}
\newtheorem{theorem}{Theorem}[section]

\newtheorem{proposition}[theorem]{Proposition}
\newtheorem{lemma}[theorem]{Lemma}

\newtheorem{corollary}[theorem]{Corollary}

\theoremstyle{definition}
\newtheorem{definition}[theorem]{Definition}

\newtheorem{remark}[theorem]{Remark}

\newcommand{\nc}{\newcommand}
\nc{\dmo}{\DeclareMathOperator}

\nc{\Q}{\mathbb{Q}}
\nc{\F}{\mathbb{F}}
\nc{\R}{\mathbb{R}}
\nc{\Z}{\mathbb{Z}}
\nc{\C}{\mathbb{C}}
\nc{\Ell}{\mathcal{L}}
\nc{\M}{\mathcal{M}}
\nc{\K}{\mathcal{K}}
\nc{\disk}{\mathbb{D}}
\nc{\hyp}{\mathbb{H}}

\nc{\CP}{\mathbb{CP}}
\nc{\cS}{\mathcal{S}}
\dmo{\Mod}{Mod}
\dmo{\Diff}{Diff}
\dmo{\Homeo}{Homeo}
\dmo{\dist}{dist}
\dmo\BDiff{BDiff}
\dmo\SO{SO}
\dmo\Hom{Hom}
\dmo\SL{SL}
\dmo\Sp{Sp}
\dmo\rank{rank}
\dmo\sig{sig}
\dmo\Out{Out}
\dmo\Aut{Aut}
\dmo\Inn{Inn}
\dmo\GL{GL}
\dmo\PSL{PSL}
\dmo\be{\beta}
\dmo\tr{tr}
\dmo\BHomeo{BHomeo}
\dmo\EHomeo{EHomeo}
\dmo\EDiff{EDiff}
\nc\Sig{\Sigma}
\dmo\Teich{Teich}
\dmo\Fix{Fix}
\dmo{\al}{\alpha}
\nc{\pair}[1]{\langle #1 \rangle}
\nc{\abs}[1]{\left| #1 \right|}
\nc{\action}{\circlearrowright}
\nc{\norm}[1]{\left | \left | #1 \right | \right |}
\nc{\abcd}[4]{\left(\begin{array}{cc} #1 & #2 \\ #3 & #4 \end{array}\right)}
\nc{\into}\hookrightarrow
\dmo{\Isom}{Isom}
\nc{\normal}{\vartriangleleft}
\dmo{\Vol}{Vol}
\dmo{\im}{Im}
\dmo{\Push}{Push}
\dmo{\Conf}{Conf}
\dmo{\PConf}{PConf}
\dmo{\RConf}{RConf}
\dmo{\id}{id}
\dmo{\Cent}{Cent}
\dmo{\Emb}{Emb}
\dmo{\cd}{\cdots}
\dmo{\hra}{\hookrightarrow}
\dmo{\ld}{\ldots}
\dmo{\ra}{\rightarrow}
\dmo{\ti}{\times}
\dmo{\Ga}{\Gamma}
\dmo{\re}{\mathbb R}
\dmo{\sbs}{\subset}
\dmo{\pa}{\partial}
\dmo{\bs}{\backslash}
\dmo{\ga}{\gamma}
\dmo{\si}{\sigma}
\newcommand{\ca}[1]{\mathcal{#1}}
\newcommand{\bb}[1]{\mathbb{#1}}

\renewcommand{\epsilon}{\varepsilon}
\nc{\coloneq}{\mathrel{\mathop:}\mkern-1.2mu=}
\nc{\margin}[1]{\marginpar{\scriptsize #1}}
\nc{\para}[1]{\medskip\noindent\textbf{#1.}}

\title{On the non-realizability of braid groups by diffeomorphisms}

\author{Nick Salter \and Bena Tshishiku}
\email{nks@math.uchicago.edu \and benatshi@stanford.edu}
\date{\today}

\address{Department of Mathematics\\ University of Chicago\\ 5734 S. University Ave., Chicago, IL 60637}
\address{Department of Mathematics\\ Stanford University\\ 480 Serra Mall Bldg.\ 380, Stanford, CA 94305}

\begin{document}

%MATH SUBJECT CLASSIFICATION: 
%MSC2010: Obstructions to group actions
%MSC2010: Braid groups; Artin groups
%MSC2010: Group actions in low dimensions

% some further directions to do: 
% -- is there any leverage for lifting problems to special subgroups of diffeomorphisms (symplectic, volume preserving, real analytic, etc)

\maketitle

\begin{abstract}
For every compact surface $S$ of finite type (possibly with boundary components but without punctures), we show that when $n$ is sufficiently large there is no lift $\sigma$ of the surface braid group $B_n(S)$ to $\Diff(S,n)$, the group of diffeomorphisms preserving $n$ marked points and restricting to the identity on the boundary. Our methods are applied to give a new proof of Morita's non-lifting theorem in the best possible range. These techniques extend to the more general setting of spaces of codimension-$2$ embeddings, and we obtain corresponding results for spherical motion groups, including the string motion group. 
\end{abstract} 

%\tableofcontents

\section{Introduction}

Let $N^k$ and $M^{k+2}$ be smooth manifolds. For any $n \ge 1$ the symmetric group $S_n$ acts on  the space $\Emb_n(N, M)$ of $C^1$ embeddings $\coprod_nN\rightarrow M$ by permuting the components of $\coprod_nN$. The quotient $\Conf_n(N,M)=\Emb_n(N,M)/S_n$ is the \emph{configuration space}. The most familiar setting is for $k = 0$, so that $M=S$ is a surface and $N=\{*\}$ is a point. In this case $\Conf_n(\{*\},S)=\Conf_n(S)$ is the configuration space of $n$-tuples of distinct, unordered points on $S$, and $\pi_1\big(\Conf_n(M)\big)=:B_n(S)$ is a {\em surface braid group}. 

The group of $C^1$ diffeomorphisms\footnote{All diffeomorphisms considered in this paper will be orientation-preserving. Also, all diffeomorphisms are $C^1$ unless otherwise noted.} $\Diff(M)$ acts on $\Conf_n(N,M)$ with the stabilizer of $[\phi]$ denoted $\Diff(M,[\phi])$. Associated to this action is a homomorphism 
\[
\mathcal P: \pi_1\big(\Conf_n(N,M)\big) \to \pi_0\big(\Diff(M, [\phi])\big)
\]
generalizing the {\em point-pushing map} $\mathcal P: B_n(S) \to \Mod(S,n)$ in the surface braid group setting. See Theorem \ref{theorem:birman} and Proposition \ref{proposition:ringpush} for detailed constructions. This note focuses on the {\em non-realizability} of $\mathcal P$ by $C^1$ diffeomorphisms. We say that $\mathcal P$ is {\em realized by ($C^1$) diffeomorphisms} if there exists a homomorphism $\sigma: \pi_1\big(\Conf_n(N,M)\big) \to \Diff(M, [\phi])$ such that the composition
\[
\pi_1\big(\Conf_n(N,M)\big) \xrightarrow\sigma \Diff(M, [\phi]) \to \pi_0\big(\Diff(M, [\phi])\big)
\]
is equal to $\mathcal P$. Such a $\sigma$, if it exists, is called a {\em lift} of $\mathcal P$.

Bestvina--Church--Souto \cite{bcs} show by a cohomological argument that $B_n(S)$ is not realized by diffeomorphisms when $S$ is closed, $\text{genus}(S)\ge2$, and $n\ge1$ $\big($note that $B_1(S)\cong\pi_1(S)\big)$. It does not seem that their methods extend to surfaces with boundary or to surfaces of low genus. In particular, this leaves the case of the classical braid group $B_n = B_n(\disk^2)$ unresolved. 

Morita's non-lifting theorem \cite{morita_nonlifting} shows that there is no lift of $\Mod(\Sigma_g) = \pi_0\big(\Diff(\Sigma_g)\big)$ to the group of $C^2$ diffeomorphisms $\Diff^2(\Sigma_g)\sbs\Diff(\Sigma_g)$ by showing that $H^*(\Mod(\Sigma_g)) \to H^*\big(\Diff^2(\Sigma_g)\big)$ fails to be injective for $g$ sufficiently large. It is tempting to try and follow this strategy for $B_n$, exploiting the fact that $B_n = \pi_0\big(\Diff(\disk^2, n)\big)$. However, there is evidence that this approach will not work, as Nariman \cite{nariman} has shown that $H^*(B_n;\Z)$ is a direct summand of $H^*(\Diff(\disk^2\setminus X_n);\Z)$, where $X_n\sbs\disk^2$ is a set of $n$ distinct points and $\Diff(\disk^2\setminus X_n)$ is the group of compactly supported diffeomorphisms of $\disk^2\setminus X_n$. 
 %so no argument in the spirit of Morita will work. 
We are able to sidestep these difficulties by using more geometric methods.

\begin{theorem}\label{theorem:surfaces}
Let $S$ be a compact surface. If $\partial S = \varnothing$ then $\mathcal P:B_n(S)\rightarrow\Mod(S,n)$ is not realized by $C^1$ diffeomorphisms for all $n \ge 6$. In the case $\partial S \ne \varnothing$, this can be improved to all $n \ge 5$.
\end{theorem}

In Section \ref{section:morita}, we use Theorem \ref{theorem:surfaces} to give a new proof of Morita's non-lifting theorem. 

\begin{theorem}\label{theorem:morita}
Let $\Sigma_g$ be a closed surface of genus $g$. For $g \ge 2$, there is no homomorphism $\Mod(\Sigma_g) \to \Diff(\Sigma_g)$ which splits the natural projection $\Diff(\Sigma_g) \to \Mod(\Sigma_g)$. %Or more succinctly (since we defined the terminology already: "mod is not realized by C^1 diffeos")
\end{theorem}

Morita's original argument \cite{morita_nonlifting} showed there is no splitting $\Mod(\Sigma_g)\to \Diff^2(\Sigma_g)$ for $g\ge18$. This was improved by Franks--Handel \cite{frankshandel}, who obtained the nonlifting theorem for $C^1$ diffeomorphisms and $g\ge3$; see also Bestvina--Church--Souto \cite[Theorem 1.2]{bcs}. Theorem \ref{theorem:morita} provides a further improvement, giving the best possible genus bound while avoiding the dynamical machinery lurking in the proof of Franks--Handel.

\begin{remark} \label{remark:homeo} Much less is understood about realizing $B_n(S)$ by homeomorphisms. Thurston showed that $B_3$ {\em is} realized by homeomorphisms \cite{thurston_mo}. In contrast, $B_6(S^2)$ is not realized by homeomorphisms (for otherwise, one could lift this realization to the branched cover $\Sigma_2\rightarrow S^2$ to obtain a realization of $\Mod(\Sigma_2)$ by homeomorphisms, and this is impossible by work of Markovic--\v{S}ari{\'c} \cite{markovicsaric}, building on the ideas of Markovic \cite{markovic}). 
\end{remark}

Along with surface braid groups, we will also be concerned with the space $\Conf_n(S^k,M)$ of configurations of unlinked, codimension-2 spheres in $M\in\{\R^{k+2},S^{k+2}\}$ for $k\ge1$. The fundamental group $B_n(S^k,M)  = \pi_1\big(\Conf_n(S^k,\re^{k+2})\big)$ is called the \emph{spherical motion group}. In the case $k=1$, this group is closely related to the {\em ring group} studied by Brendle and Hatcher in \cite{brendle_hatcher} (see Section \ref{section:extension}). The main result is as follows. %In Section \ref{section:sphereproof}, we prove the following. 

\begin{theorem}\label{theorem:spheres}
Let $M$ be $S^{k+2}$ or $\R^{k+2}$. Fix an unlinked embedding $\phi:\coprod_nS^k\hookrightarrow M$, and let $[\phi]\in\Conf_n(S^k,M)$ denote the corresponding configuration. Let $\ca D(M,[\phi])\le\Diff(M)$ be the group of compactly-supported $C^1$ diffeomorphisms isotopic the identity and such that $[f\circ\phi]=[\phi]$. If either
\begin{enumerate}[(a)]
\item $M=\R^{k+2}$ and $n\ge5$, or 
\item $M=S^{k+2}$ and $n\ge6$, 
\end{enumerate} then the ``spherical push map'' $\mathcal P: B_n(S^k,M) \to \pi_0\big(\ca D(M,\phi)\big)$ is not realized by diffeomorphisms.
\end{theorem}

\begin{remark}
The arguments of Theorems \ref{theorem:surfaces} and \ref{theorem:spheres} can be extended to certain finite-index subgroups, but do not work, e.g. for the pure braid group $PB_n \le B_n(\disk^2)$. It is also not clear whether the bounds in Theorem \ref{theorem:surfaces} or Theorem \ref{theorem:spheres} can be improved, although the methods of the current paper do not extend beyond the stated ranges. See Remark \ref{remark:homeo} for some related discussions.
\end{remark}

In Theorem \ref{theorem:spheres}, the diffeomorphism groups under consideration are required to fix the image of $\phi$ pointwise up to permutation. In Section \ref{section:extension}, we use work of Parwani \cite{parwani} to give an extension of Theorem \ref{theorem:spheres} that deals with the possibility of a lift of $\mathcal{P}$ that only fixes the image of $\phi$ {\em setwise}, in the case $k = 1$. We also treat a generalization of Theorem \ref{theorem:surfaces}, where the marked points on $S$ are replaced by boundary components.\\

The proof of Theorems \ref{theorem:surfaces} and \ref{theorem:spheres} involves two main ingredients. The first is the Thurston stability theorem \cite{thurston_stability}, which can be used to impose restrictions on the homology of finitely-generated subgroups of diffeomorphisms. The second is the fact that $B_n$ interacts poorly with these restrictions. The main theorems are proved by exhibiting suitable subgroups closely related to $B_n$ in each of the braid or motion groups under consideration.

The paper is organized as follows. In Section \ref{section:surfacepush}, we briefly review Birman's theory of push maps for surface braid groups. In Sections \ref{section:surfaceproof} and \ref{section:morita} we prove Theorems \ref{theorem:surfaces} and \ref{theorem:morita}, respectively. In Section \ref{section:spherepush} we develop a notion of push maps for spherical motion groups. In Section \ref{section:sphereproof} we prove Theorem \ref{theorem:spheres}. Finally in Section \ref{section:extension}, we prove some strengthenings of Theorems \ref{theorem:surfaces} and \ref{theorem:spheres} in low dimensions. 
\vspace{.1in} 

\noindent{\bf Acknowledgements.} The authors wish to thank their advisor B.\ Farb for his guidance and support and for extensive comments on drafts of this paper. The authors express their gratitude to the anonymous referee for numerous improvements, and in particular for identifying the suitability of our methods for giving a new proof of the Morita non-lifting theorem. The authors thank I.\ Agol for remarking to them that $B_6(S^2)$ is not realized by homeomorphisms and A.\ Hatcher for suggesting the proof of Proposition \ref{proposition:hasbraids}. Finally, the authors thank J.\ Bowden, A.\ Hatcher, D.\ Margalit, and A. Putman for several valuable comments. 

\section{From configuration spaces to mapping class groups}\label{section:surfacepush}

In this section, we review how surface braid groups give rise to subgroups of mapping class groups via {\em push maps}.  Let $S$ be a surface. The {\em pure configuration space of $n$ points in $S$} is defined as
\[
\PConf_n(S) = \{(x_1, \dots, x_n) \mid x_i\in\text{int}(S)\text{ and } x_i \ne x_j \mbox{ if } i \ne j\}.
\]
The {\em configuration space} is defined as the quotient $\Conf_n(S) = \PConf_n(S) / S_n$ by the (free) action of the symmetric group on $n$ letters via permutation of coordinates. 
\begin{definition}
The {\em braid group on $n$ strands in $S$}, written $B_n(S)$, is defined to be $\pi_1\big(\Conf_n(S)\big)$. In the case $S = \disk^2$, we write $B_n = B_n(\disk^2)$. 
\end{definition}

The following is due to J. Birman. See \cite[Section 9.1.4]{fm}. 

\begin{theorem}[Birman]\label{theorem:birman}
Let $S$ be a compact surface with possibly nonempty boundary. Let $X_n = \{x_1, \dots, x_n\}$ be a set of $n$ distinct points in $S$. There is a homomorphism
\[
\mathcal P: B_n(S) \to \pi_0\big(\Diff(S, \partial S, X_n)\big);
\]
here $\Diff(S, \partial S, X_n)$ is the group of $C^1$ diffeomorphisms of $S$ restricting to the identity on $\partial S$ that preserve $X_n$ setwise. The kernel of $\mathcal P$ is isomorphic to a quotient of $\pi_1\big(\Diff(S, \partial S)\big)$. 
\end{theorem}

\begin{remark}\label{remark:nocontain}
The condition $\pi_1\big(\Diff(S, \partial S)\big) = 1$ is satisfied whenever $\chi(S) < 0$, and also when $S = \disk^2$ (see \cite{ee} and \cite{es}). In the exceptional cases, $\pi_1\big(\Diff(S^2)\big) \cong \Z / 2$, and $\pi_1\big(\Diff(T^2)\big) \cong \Z^2$. It follows that for all $n \ge 5$ (the cases under consideration in this paper), the map $\mathcal P$ is nonzero.
\end{remark}

%%%%%%%%%%%%%%%%%%%%%%%%%%%%%%%%%%%%%%%%%%%%%%%%%%%
\section{Proof of Theorem \ref{theorem:surfaces}} \label{section:surfaceproof}

The situation can be expressed diagrammatically as follows:
\[
\xymatrix{												& \Diff(S, \partial S, X_n) \ar[d]^\pi\\
B_n(S) \ar[r]_>>>>{\mathcal {P}}	\ar@{-->}[ur]^{\sigma}	& \pi_0\big( \Diff(S, \partial S, X_n)\big).
}
\]
We seek to obstruct the existence of a homomorphism $\sigma$ lifting $\mathcal P$. Our method will be to reduce to the {\em Thurston stability theorem}.

\para{Step 1: Local indicability and the Thurston stability theorem}
The aim of this section is to show that certain diffeomorphism groups do not contain braid subgroups. We will be concerned with a property of groups known as {\em local indicability}.

\begin{definition}\label{definition:locindic}
A group $G$ is said to be {\em locally indicable} if every nontrivial finitely-generated subgroup $\Gamma \le G$ admits a surjection $\Gamma \to \Z$. Equivalently, $G$ is locally indicable if every finitely-generated subgroup $\Gamma$ has $H^1(\Gamma, \R) \ne 0$.

A group $G$ is said to be {\em strongly non-indicable} if there exists a nontrivial finitely-generated subgroup $\Gamma$ that is perfect, i.e.\ with $[\Gamma, \Gamma] = \Gamma$. 
\end{definition}

\begin{remark}\label{remark:quotients}
Suppose $G$ is not locally indicable, and let $H \le G$ be a subgroup witnessing this fact. If $N \normal G$ is a normal subgroup with $H \cap N \ne H$, then $HN / N$ witnesses the non-indicability of $G/ N$. The same is true for strong non-indicability. 
\end{remark}

In \cite{thurston_stability}, Thurston showed that certain diffeomorphism groups are locally indicable. %While the statement below is concerned only with surfaces, the analogous statement holds for all smooth manifolds $M$ of dimension $k \ge 1$. 

\begin{theorem}[Thurston stability theorem] \label{theorem:thurston_stability}
Let $M$ be a manifold, and let $x \in M$ be given. For a diffeomorphism $g$ of $M$ fixing $x$, we write $(Dg)_x \in GL(T_x M)$  for the derivative. Then the group
\[
\mathcal G = \{g \in \Diff(M) \mid g(x) = x,\> (Dg)_x =I \}
\]
is locally indicable (and hence any subgroup of $\mathcal{G}$ is locally indicable as well). 
\end{theorem}

The strategy for the remainder of the proof is to argue that a lift $\sigma$ of $\mathcal P$ would force $\mathcal G$ to contain a non-locally-indicable subgroup. We will show that $B_n$ is a suitable group. 

\para{Step 2: Braid groups are strongly non-indicable}
\begin{proposition}\label{theorem:gorinlin}\mbox{ }
\begin{enumerate}[(i)]
\item For $n \ge 5$, the set
\[
S = \{\sigma_i \sigma_{i+1}^{-1} \mid 1 \le i \le n-2 \}
\]
generates $[B_n, B_n]$. Moreover, the elements of $S$ are all mutually conjugate within $[B_n, B_n]$.
\item \emph{(Gorin--Lin)} For $n \ge 5$, the commutator subgroup of the braid group $B_n$ is {\em perfect}, i.e.\ $[B_n,B_n]=[[B_n,B_n],[B_n,B_n]]$.
\end{enumerate} 
\noindent Consequently $B_n$ is strongly non-indicable for $n \ge 5$.
\end{proposition}

\begin{proof}
We begin with the proof of $(i)$. For $1 \le i \le n$, let $\sigma_i\in B_n$ denote the braid that passes the $i^{th}$ strand over the $(i+1)^{st}$, with subscripts interpreted mod $n$. The elements $\sigma_1, \dots, \sigma_n$ are all mutually conjugate, and the abelianization map $A: B_n \to \Z$ is given by the total exponent sum of all the generators. Consequently, the set
\[
S = \{ \sigma_i \sigma_{i + 1}^{-1} \mid 1 \le i \le n-1\}
\]
normally generates $[B_n, B_n]$ inside $B_n$. 

To prove the claim, it therefore suffices to show that the subgroup $\pair{S}$ of $B_n$ generated by $S$ is normal, which in turn reduces to showing that $\sigma_j (\sigma_i \sigma_{i+1}^{-1}) \sigma_j^{-1} \in \pair{S}$ for any $1 \le j \le n$. As $n \ge 5$, the generator $\sigma_{i+3}$ commutes with $\sigma_i$ and $\sigma_{i+1}$, from which
\[
\sigma_j (\sigma_i \sigma_{i+1}^{-1}) \sigma_j^{-1}  = (\sigma_j \sigma_{i+3}^{-1})(\sigma_i \sigma_{i+1}^{-1}) (\sigma_j \sigma_{i+3}^{-1})^{-1}.
\]
The right-hand side exhibits $\sigma_j (\sigma_i \sigma_{i+1}^{-1}) \sigma_j^{-1} $ as a product of elements of $\pair{S}$, and the result follows.

The next step is to show that the elements of $S$ are all conjugate within $[B_n, B_n]$. Via the braid relations,
\begin{equation}\label{equation:relation}
(\sigma_i \sigma_{i+1} \sigma_{i+2}) \sigma_i \sigma_{i+1}^{-1} (\sigma_i \sigma_{i+1} \sigma_{i+2}) ^{-1} = \sigma_{i+1} \sigma_{i+2}^{-1}.
\end{equation}
As above, the element
\[
\sigma_i \sigma_{i+1} \sigma_{i+2}\sigma_{i+3}^{-3} \in [B_n, B_n]
\]
also conjugates $\sigma_i \sigma_{i+1}^{-1}$ to $\sigma_{i+1} \sigma_{i+2}^{-1}$.

From what has been established above, to establish $(ii)$, it is sufficient to express each $\sigma_i \sigma_{i+1}^{-1}$ as a commutator in $[B_n, B_n]$. For $n \ge 5$, there is some $j$ for which $\sigma_j$ commutes with both $\sigma_i$ and $\sigma_{i+1}$, and therefore the expression
\[
\sigma_i \sigma_{i+1}^{-1} = [\sigma_{i+1} \sigma_i \sigma_j^{-2},\sigma_{i+1} \sigma_{j}^{-1}]
\]
(which holds as a result of the braid relations) proves the claim.
\end{proof}

\begin{remark}
In fact, $[B_n,B_n]$ is finitely generated for all $n \ge 2$. We content ourselves with the given proof because it is better suited to the applications in the present paper.
\end{remark}

%Question: is Mod(S^2, n) defined elsewhere?
%Question: where best to remark that the sphere special case is still okay?
\para{Step 3: Produce $B_m \le B_n(S)$}
The following is implied by a theorem of Paris-Rolfsen \cite[Theorem 4.1(iii)]{parisrolfsen}.
\begin{theorem}[Paris-Rolfsen]\label{theorem:parisrolfsen}
For $S \ne S^2$, the inclusion of subsurfaces $(\disk, X_{n}) \into (S, X_{n})$ induces an injective map $B_{n} \into B_{n}(S)$. In the case $S = S^2$, an inclusion $(\disk, X_n) \into (S^2, X_{n+1})$ induces a homomorphism $B_n \to B_{n+1}(S^2)$. The kernel of this homomorphism is contained in the cyclic group $\pair{\Delta}$ generated by the Dehn twist of a boundary-parallel curve using the identification $B_n \cong \Mod(\disk, n)$, and is contained in the center of $B_n$. 
\end{theorem}

\begin{remark}\label{remark:stabilizers}
By construction, the subgroup $B_{n-1} \le B_{n}$ stabilizes $X_{n} \setminus X_{n-1}$. More precisely, if $\tau\in B_{n-1}\le B_{n}(S)$ and $\phi\in\Diff(S, \partial S, X_{n})$ is any representative of $\ca P(\tau)\in\pi_0\big(\Diff(S, \partial S, X_{n})\big)$, then $\phi$ fixes the point $X_{n}\setminus X_{n-1}$. Similarly, the image of $B_{n}$ inside $B_{n+1}(S^2)$ stabilizes $X_n \setminus X_{n-1}$.
% if $\phi \in \Diff(S, \partial S, X_{n+k})$ is any representative of $[\phi] \in B_n \le B_{n+k}(S)$, then $\phi$ fixes each element of $X_{n+k} \setminus X_n$.
\end{remark}

\para{Step 4: Reduction to Thurston stability}
In order to apply the Thurston stability theorem, we must first study the derivative mapping at the global fixed point.
\begin{lemma}\label{lemma:nonabelian}
For $n \ge 5$, every homomorphism $f: B_n \to GL^+_2(\R)$ has abelian image.
\end{lemma}

This is a consequence of the following more general criterion (which we will employ again in Section \ref{section:extension}).

\begin{lemma}\label{lemma:general}
Let $G$ be a group generated by elements $\tau_1, \dots \tau_n$ that satisfy the following properties:
\begin{enumerate}
\item The elements $\tau_i$ are all mutually conjugate.
\item There exists $k \ge 1$ such that $[\tau_i, \tau_j] = 1$ for $\abs{j-i} \ge k$ (here we mean distance in $\R / n \Z$). 
\end{enumerate}
Then for $n \ge 2k+1$, every homomorphism $f: G \to \GL_2^+(\R)$ has abelian image.
\end{lemma}

\begin{proof}
It suffices to show that the projection $\bar f: G \to \GL_2^+(\R) \to \PSL_2(\R)$ has image contained in a one-parameter subgroup. This is because the preimage in $\GL_2^+(\R)$ of any one-parameter subgroup in $\PSL_2(\R)$ is abelian. For convenience, we will write $\bar \tau_i$ in place of $\bar f(\tau_i)$. By condition (1) above, if $\bar f$ is a nontrivial homomorphism, then each $\bar \tau_i \ne I$. 

If the image of $\bar f$ is not contained in some one-parameter subgroup, then in particular, there must be some pair of elements $\bar \tau_i$ and $\bar \tau_j$ that do not commute. By relabeling if necessary, we may assume $i = 1$ and $2\le j\le k$. Furthermore, we may assume $j$ is the smallest integer between 2 and $k$ for which $\bar \tau_1$ and $\bar \tau_j$ do not commute. 

We wish to show $j = 2$. Suppose $j > 2$. If $\bar \tau_{j-1}$ and $\bar \tau_j$ do not commute, then by relabeling again, we may assume that $\bar \tau_1$ and $\bar \tau_2$ do not commute. If, on the other hand, $\bar \tau_{j-1}$ commutes with $\bar \tau_j$, then both $\bar \tau_1$ and $\bar \tau_j$ are contained in the centralizer $C_{\PSL_2(\R)}(\bar \tau_{j-1})$. As the latter is a one-parameter subgroup, necessarily $\bar \tau_1$ and $\bar \tau_j$ commute, contrary to assumption. We conclude that up to a cyclic relabeling of the generators $\tau_i$, we must have $\bar \tau_1$ and $\bar \tau_2$ noncommuting elements of $\PSL_2(\R)$. 

By condition (2) above and the assumption $n \ge 2k+1$, the element $\bar \tau_{k+2}$ commutes with both $\bar \tau_1$ and $\bar \tau_2$. Therefore, $\bar \tau_1$ and $\bar \tau_2$ are contained in the abelian subgroup $\Cent_{\PSL_2(\R)}(\bar \tau_{k+2})$, contrary to assumption.
\end{proof}

\begin{proof}(of Lemma \ref{lemma:nonabelian})
We show that $B_n$ satisfies the hypotheses of Lemma \ref{lemma:general} for $k = 2$. Indeed, for $1 \le i \le n$, let $\tau_i = \sigma_i$, the $i^{th}$ standard generator of $B_n$. We interpret $\sigma_n$ to be the element crossing the $n^{th}$ strand over the first, under a cyclic ordering of the strands. As the elements $\sigma_i$ are mutually conjugate and $[\sigma_i, \sigma_j] = 1$ for $\abs{j-i} \ge 2$, the result follows. 
\end{proof} 

\begin{remark}
The assumption $n \ge 5$ in Lemma \ref{lemma:nonabelian} cannot be relaxed: it is well-known that there is a homomorphism $B_3 \to \SL_2 \Z$ with nonabelian image. The case $n = 4$ follows from the existence of an exceptional surjective homomorphism $B_4 \to B_3$. 
\end{remark}

To complete the proof of Theorem \ref{theorem:surfaces}, we begin with the case $\partial S = \varnothing$. Suppose, for a contradiction, that a lift $\sigma: B_n(S) \to \Diff(S, \pa S, X_n)$ (for $n \ge 6$) is given. By Theorem \ref{theorem:parisrolfsen}, there is a nontrivial homomorphism $B_{n-1} \to \Mod(S,n)$; it follows from Remark \ref{remark:quotients} that $\Mod(S, n)$ is strongly non-indicable. By Remark \ref{remark:stabilizers}, the lift $\sigma(B_{n-1})$ fixes some point $x\in X_n \setminus X_{n-1}$. Let $D: B_{n-1} \to \GL_2^+(\R)$ denote the derivative mapping at $x$. Via Lemma \ref{lemma:nonabelian}, $[B_{n-1}, B_{n-1}] \le \ker D$. Thurston stability (Theorem \ref{theorem:thurston_stability}) then asserts that $[B_{n-1}, B_{n-1}]$ must be locally indicable, but this contradicts Theorem \ref{theorem:gorinlin}. 

To obtain the improvement $n \ge 5$ in the case $\partial S$ is non-empty, we simply apply the preceding arguments to any point $x \in \partial S$. Here, we do not need to pass to $B_{n-1}$ in order to produce a fixed point a la Remark \ref{remark:stabilizers}, and so the argument applies for all $n \ge 5$. 
\qed

%%%%%%%%%%%%%%%%%%%%%

\section{The Morita non-lifting theorem}\label{section:morita}

The purpose of this section is to show how the methods of Theorem \ref{theorem:surfaces} can be extended to give a new proof of Morita's non-lifting theorem. We are grateful to the referee for observing that our methods should be applicable to this situation, and for suggesting Steps 1 and 2 below.

%Comment: I will need to go through and systematize our notation for indicating the regularity of the diffeomorphisms.
%Comment: Also, the intro needs to be re-worked to include mention of this theorem!

%Comment: I will need to create the following figure(s): a picture of the hyperelliptic involution, equipped with the 2g+1 curves needed to find a hyperelliptic braid group.
\begin{proof}[Proof of Theorem \ref{theorem:morita}] Suppose that there is a realization $\sigma: \Mod(\Sigma_g) \to \Diff(\Sigma_g)$. We will arrive at a contradiction. The argument is divided into four steps.

\noindent {\bf Step 1: A large subgroup with a finite orbit.} In this step, we indicate a particular constraint on the dynamics of any realization of the mapping class group by diffeomorphisms. Let $\iota$ denote the hyperelliptic involution (as depicted in Figure \ref{figure:hyperelliptic}). Let $C(\iota)$ denote the centralizer of $\iota$ inside $\Mod(\Sigma_g)$. 

\begin{lemma}\label{lemma:fixset}
For any realization $\sigma$, the fixed set $\Fix(\sigma(\iota))$ consists of exactly $2g+2$ points.
\end{lemma}

%Comment: add a more precise reference to farb-margalit
\begin{proof}
This is a standard argument that follows from the Lefschetz fixed-point theorem. See \cite[Section 7.1.2]{fm} for details.
\end{proof}

A standard principle in the theory of group actions gives the following corollary.

\begin{corollary}\label{corollary:finiteorbit}
The subgroup $\sigma(C(\iota)) \le \Diff(\Sigma_g)$ preserves the finite set $\Fix(\sigma(\iota))$; associated to this is a permutation representation $\rho: C(\sigma(\iota)) \to S_{2g+2}$, the symmetric group on $2g+2$ letters.
\end{corollary}

\noindent{\bf Step 2: A non-indicable subgroup of $C(\iota)$.} 
\begin{lemma}\label{lemma:hyperellipticbraids}
For all $g \ge 2$, $C(\iota)$ contains a strongly non-indicable subgroup $B$ isomorphic to a quotient of $B_{2g+2}$.
\end{lemma}

\begin{proof}
Consider the family of $2g+1$ simple closed curves $c_1, \dots, c_{2g+1}$ indicated in Figure \ref{figure:hyperelliptic}. As the geometric intersection $i(c_i, c_{i+1}) = 1$ for all $i$, and $i(c_i, c_j) = 0$ for $\abs{i-j} \ge 2$, the subgroup $B \le \Mod(\Sigma_g)$ generated by the Dehn twists $T_{c_i}$ satisfy the braid relations: $B$ is a (nontrivial) quotient of $B_{2g+2}$. It follows from Theorem \ref{theorem:gorinlin} and Remark \ref{remark:quotients} that $B$ is strongly non-indicable. 

As each $c_i$ is invariant under the action of $\iota$, it follows that each $T_{c_i} \in C(\iota)$; consequently $B \le C(\iota)$ as claimed.
\end{proof}

\begin{remark}\label{remark:Bsubgroup}
Let $B'$ denote the image of $B_{2g+1}$ in $B$. By the above arguments, $B'$ is also strongly non-indicable for $g \ge 2$. 
\end{remark}

\noindent{\bf Step 3: The action of $B_{2g+2}$ on $\Fix(\sigma(\iota))$.} In this step, we explicitly identify the action of $\sigma(B)$ on $\Fix(\sigma(\iota))$.

\begin{lemma}\label{lemma:action}
There is a commutative diagram
\[
\xymatrix{
B_{2g+2} \ar[r]^\mu \ar[d] & S_{2g+2}\\
B \ar[ur]_{\rho \circ \sigma}
}
\]
where the map $\mu: B_{2g+2} \to S_{2g+2}$ is the canonical permutation homomorphism. Letting $B' \le B$ be the subgroup defined in Remark \ref{remark:Bsubgroup}, it follows that the action of $B'$ on $\Fix(\sigma(\iota))$ has a global fixed point. 
\end{lemma}

\begin{proof}  
%1) For each generator of $B_{2g+2}$, there is a representative diffeo inducing the expected permutation.
Let $\sigma_i \in B_{2g+2}$ denote the standard generator of $B_{2g+2}$ interchanging strands $i$ and $i+1$, so that $\mu(\sigma_i) = (i\ i+1)$. The homomorphism $B_{2g+2} \to B$ sends $\sigma_i$ to the Dehn twist $T_{c_i}$ indicated in Figure \ref{figure:hyperelliptic}. Let $\widetilde T_{c_i}$ denote a realization of this Dehn twist supported on a neighborhood of $c_i$ invariant under $\sigma(\iota)$. Then $\rho(\widetilde T_{c_i})$ is the involution $(i\ i+1) = \mu(\sigma_i) \in S_{2g+2}$. 

We next claim that if $\alpha, \alpha' \in C(\sigma(\iota))$ are isotopic, then $\rho(\alpha) = \rho(\alpha')$. Modulo the claim the result follows easily, since by the above paragraph each element of $B$ has some representative diffeomorphism (obtained by taking a suitable product of the $\widetilde T_{c_i}$) inducing the expected permutation.

The claim is most easily established by temporarily leaving concerns of smoothness behind. Let $A \subset \Homeo(\Sigma_g)$ denote the isotopy class of $\alpha, \alpha'$ within $\Homeo(\Sigma_g)$. Letting $C_{\Homeo(\Sigma_g)}(\sigma(\iota))$ denote the centralizer of $\sigma(\iota)$ within $\Homeo(\Sigma_g)$, observe that $\rho$ extends to a homomorphism $ \rho: C_{\Homeo(\Sigma_g)}(\sigma(\iota)) \to S_{2g+2}$. 

We claim that as a map of topological spaces (endowing $C_{\Homeo(\Sigma_g)}(\sigma(\iota))$ with the compact-open topology and $S_{2g+2}$ with the discrete topology), $\rho$ is continuous. Let $\phi \in C_{\Homeo(\Sigma_g)}(\sigma(\iota))$ and $x \in \Fix(\sigma(\iota))$ be given. Let $U \subset \Sigma_g$ be an open neighborhood such that $U \cap \Fix(\sigma(\iota)) = \{\phi(x)\}$. If $\psi \in C_{\Homeo(\Sigma_g)}(\sigma(\iota))$ is sufficiently close to $\phi$, then $\psi(x) \in U$, but as $\psi(x) \in \Fix(\sigma(\iota))$, it follows that $\psi(x) = \phi(x)$.

To establish the claim, it therefore suffices to show that $\alpha$ and $\alpha'$ lie in the same connected component of $C_{\Homeo(\Sigma_g)}(\sigma(\iota))$. Proposition 9.4 of \cite{fm} asserts that if $\phi, \psi \in C_{\Homeo(\Sigma_g)}(\sigma(\iota))$ are isotopic, then there exists an isotopy through elements of $C_{\Homeo(\Sigma_g)}(\sigma(\iota))$. The claim, and hence the result, follows.
\end{proof}

\begin{figure}[h!]
\labellist 
\small\hair 2pt 
\pinlabel $\iota$ at 422 400
\pinlabel $c_1$ at 522 370
\pinlabel $c_2$ at 590 410
\pinlabel $c_4$  at 655 410 
\pinlabel $c_{2g}$  at 770 410 
\pinlabel $c_3$  at 650 380 
\pinlabel $c_{2g+1}$  at 852 370 
\endlabellist 
  \centering
\includegraphics[scale=.5]{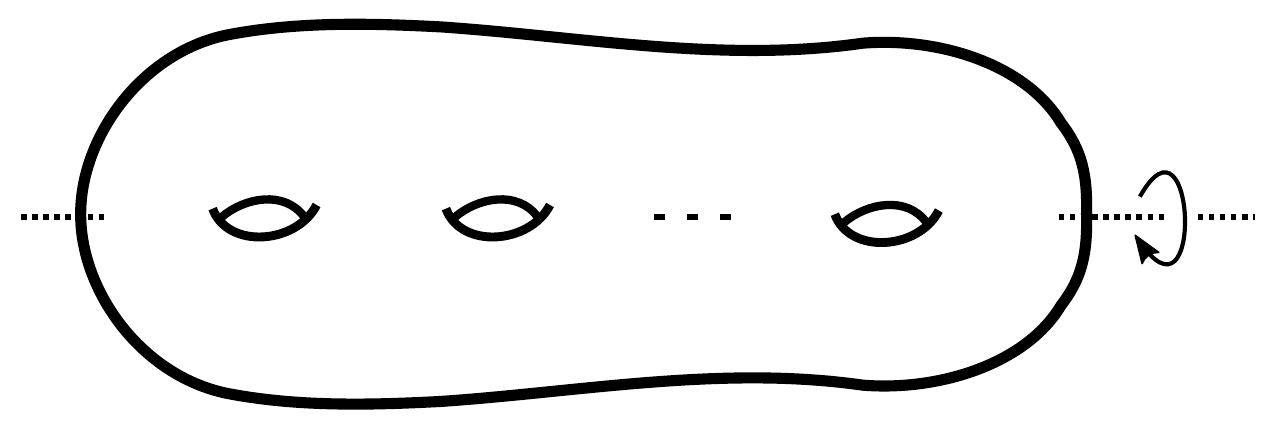}\hspace{0.5in}
\includegraphics[scale=.5]{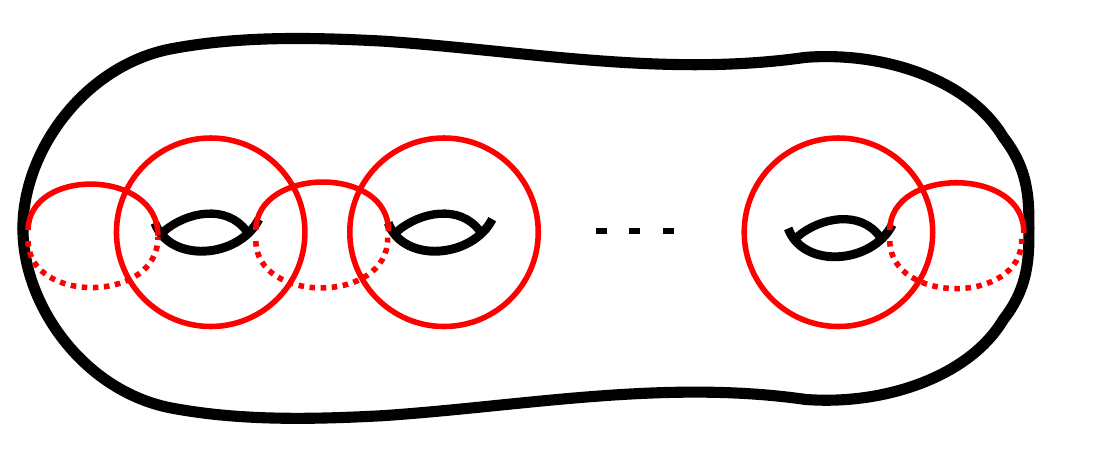}
      \caption{The hyperellitpic involution $\iota$ and curves $c_i$ whose Dehn twists $T_{c_i}$ generate a quotient of $B_{2g+2}$.}
      \label{figure:hyperelliptic}
\end{figure}
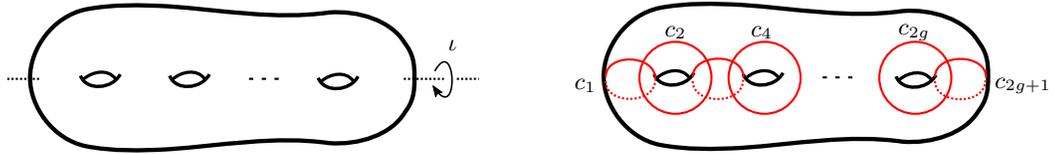

\noindent{\bf Step 4: Deriving the contradiction.}
From Steps 1 - 3 above, we have shown that if there is a realization $\sigma: \Mod(\Sigma_g) \to \Diff(\Sigma_g)$, then the strongly non-indicable subgroup $\sigma(B') \le \Diff(\Sigma_g)$ must act on $\Sigma_g$ with a global fixed point $p \in \Sigma_g$. Consider the homomorphism
\[
D_p \circ \sigma: B' \to \GL_2^+(\R).
\]
According to Lemma \ref{lemma:nonabelian}, as $B'$ is a quotient of $B_{2g+1}$, the image of $D_p \circ \sigma$ must be abelian. Letting $P \le B'$ denote any nontrivial finitely-generated perfect subgroup of $B'$, it follows that $\sigma(P)$ acts trivially on the tangent space $T_p\Sigma_g$. Theorem \ref{theorem:thurston_stability} then asserts that $P$ must be locally indicable, but this is impossible by assumption.
\end{proof}

%%%%%%%%%%%%%%%%%%%%%

\section{Push maps for spherical motion groups}\label{section:spherepush}
We turn now to Theorem \ref{theorem:spheres}. It is first necessary to establish the existence of the push homomorphisms $\mathcal P$ that are the higher-dimensional analogues of the homomorphism in Theorem \ref{theorem:birman}. Fix $k,n\ge1$. For $M=\R^{k+2}$ or $S^{k+2}$, consider the space $\Emb_n(S^k,M)$ of $C^1$ embeddings $\coprod_n S^k\to M$. The symmetric group $S_n$ acts on $\coprod_nS^k$ by permuting the components, and this induces an action on $\Emb_n(S^k,M)$ by precomposing an embedding by a permutation. Fix an embedding $\phi$ that is unlinked, and let $\Emb_{n}(S^{k},M;\phi)$ denote the path component of $\phi$. Define the \emph{configuration space }
\[\Conf_n(S^k,M)=\Emb_n(S^k,M;\phi)/S_n.\]
An element of $\Conf_n(S^k,M)$ is a collection of disjoint, unordered, unlinked spheres, each of which comes with a parameterization.

\begin{definition} \label{definition:sphericalmotion}
Let $[\phi]\in\Conf_n(S^k,M)$ denote the equivalence class of the embedding $\phi$. The group $B_n(S^k,M):=\pi_1\big(\Conf_n(S^k,M), [\phi]\big)$ is a \emph{spherical motion group}.\footnote{These groups were first studied by Dahm \cite{dahm}.}
\end{definition}

In order to state the analog of Theorem \ref{theorem:birman} for $B_n(S^k,M)$, let $\ca D(M)\le\Diff(M)$ be the group of compactly-supported $C^1$ diffeomorphisms isotopic to the identity, and let $\ca D(M,[\phi])\le\ca D(M)$ be the subgroup of diffeomorphisms that satisfy $[f\circ \phi]=[\phi]$. Viewing $\phi$ as defining a parameterization on its image $\text{Im}(\phi)\sbs M$, diffeomorphisms of $\ca D(M,[\phi])$ preserve $\text{Im}(\phi)$ together with the parameterization on each sphere, up to permutations. In particular, $f\in\ca D(M,[\phi])$ fixes \emph{pointwise} any component of $\text{Im}(\phi)$ taken to itself.

\begin{proposition}\label{proposition:ringpush}
Fix $n\ge1$. There is a homomorphism $\mathcal P : B_n(S^k,M) \to \pi_0\big(\ca D(M,[\phi])\big)$. The kernel of $\mathcal P$ is abelian. 
\end{proposition}
\begin{proof} 

There is an evaluation map $\eta: \ca D(M)\ra\Conf_n(S^k,M)$ defined by $f\mapsto [f\circ\phi]$. By Palais \cite{palais} this map determines a fibration 
\[\ca D(M,[\phi])\ra\ca D(M)\xrightarrow\eta\Conf_n(S^k,M).\]
The long exact sequence of homotopy groups of this fibration gives an exact sequence 
\[\pi_1\big(\ca D(M)\big)\ra B_n(S^k,M)\xrightarrow{\ca P}\pi_0\big(\ca D(M,[\phi])\big).\]
This defines $\ca P$. Note that as $\ca D(M)$ is a topological group, $\pi_1\big(\ca D(M)\big)$ is abelian, from which it follows that $\ker\ca P$ is as well.  
\end{proof}

%%%%%%%%%%%%%%%%%%%%%%%%%%%%%%%%%%%%%%%%%%%%%%%%%%%
\section{Proof of Theorem \ref{theorem:spheres}} \label{section:sphereproof}
Once again, the situation can be expressed diagrammatically as follows:
\[
\xymatrix{												& \ca D(M,[\phi]) \ar[d]^\pi\\
B_n(S^k,M) \ar[r]_>>>>{\mathcal {P}}	\ar@{-->}[ur]^{\sigma}		& \pi_0\big( \ca D(M,[\phi])\big).
}
\]
We seek to obstruct the existence of a lift $\sigma$ of $\mathcal P$. The outline of the proof is essentially the same as for Theorem \ref{theorem:surfaces}. We will not reproduce Step 1 of Theorem \ref{theorem:surfaces}, as the Thurston stability theroem holds for any smooth manifold. Also, 
Step 2 of Theorem \ref{theorem:surfaces}, which concerns the group theory of $B_n$, needs no modification. As such, the proof of Theorem \ref{theorem:spheres} will begin with finding a nontrivial homomorphism $B_n\ra B_n(S^k,M)$. 

\para{Step 1: Produce nontrivial $B_n\ra B_n(S^k,M)$}  In this section we prove the following proposition. 

\begin{proposition}\label{proposition:hasbraids} If $M=\R^{k+2}$, then there is an embedding $B_n\hra B_n(S^k,M)$. If $M=S^{k+2}$, then there is a homomorphism $B_n\ra B_n(S^k,M)$ whose kernel is contained in the center $Z(B_n)$.
\end{proposition}

\begin{proof}
To produce the desired homomorphism, we first find a subspace $\ca C\sbs\Conf_n(S^k,M)$ such that $\pi_1(\ca C)$ contains $B_n$. This uses work of Brendle--Hatcher. Then we will study the composition $B_n\hra\pi_1(\ca C)\ra \pi_1\big(\Conf_n(S^k,M)\big)$ by looking at the induced action of $B_n$ on $\pi_1\big(M\bs\coprod_nS^k\big)\simeq F_n$. For $M=\R^{k+2}$ this action coincides with the Artin representation $B_n\ra\text{Aut}(F_n)$, which is well-known to be injective. For $M=S^{k+2}$, we obtain instead a quotient of the Artin representation $B_n\ra\text{Out}(F_n)$, and we explain why its kernel is $Z(B_n)$.  

To define $\ca C$, give $\R^{k+2}$ coordinates $(x,y,z,w_1,\ld,w_{k-1})$ and fix an embedding $f: S^k\hra\R^{k+2}$ whose image is the sphere of radius 1 centered at the origin in $\R^{k+1}\simeq \{(x,y,z,w_1,\ld,w_{k-1}): x=0\}$. Consider the space $\ca E$ of embeddings $\phi:\coprod_nS^k\ra\R^{k+2}$ where the embedding on each component is the composition of $f$ with a dilation of $\R^{k+2}$ followed by a translation in the $xy$-plane. The quotient $\ca C=\ca E/S_n$ is a subspace of $\Conf_n(S^k,\R^{k+2})$. We also obtain $\ca C\sbs\Conf_n(S^k,S^{k+2})$ by choosing an embedding $\R^{k+2}\hra S^{k+2}$. 

There is a map $a:\ca C\ra\ca{UW}_n$ to the \emph{untwisted wicket space} of Brendle--Hatcher \cite{brendle_hatcher} obtained by restricting an embedding $\phi:\coprod_nS^k\ra\R^{k+2}$ to $\coprod_nV$, where $V\sbs S^k$ is the subspace $f^{-1}\{(0,y,z,0,\ld,0)\}\simeq S^1$. The map $a$ is a homeomorphism by the construction of $\ca C$; furthermore, $\pi_1(\ca{UW}_n)$ contains a braid group by \cite[Proposition 3.1]{brendle_hatcher}. In $\pi_1(\ca C)$, this braid group is generated by motions $\rho_1,\ld,\rho_{n-1}\in\pi_1(\ca C)$ that exchange the $i^{th}$ and $(i+1)^{st}$ spheres of a fixed embedding $\phi$, passing the $(i+1)^{st}$ sphere through the $i^{th}$ sphere. See Figure \ref{figure:motion}.

\begin{figure}[h!]
\labellist 
\small\hair 2pt 
\pinlabel $x_1$ at 72 55
\pinlabel $x_2$  at 162 55 
\pinlabel $\rho_1(x_1)$ at 797 55 
\pinlabel $\rho_1(x_2)$ at 625 67
\endlabellist 
  \centering
\includegraphics[scale=.45]{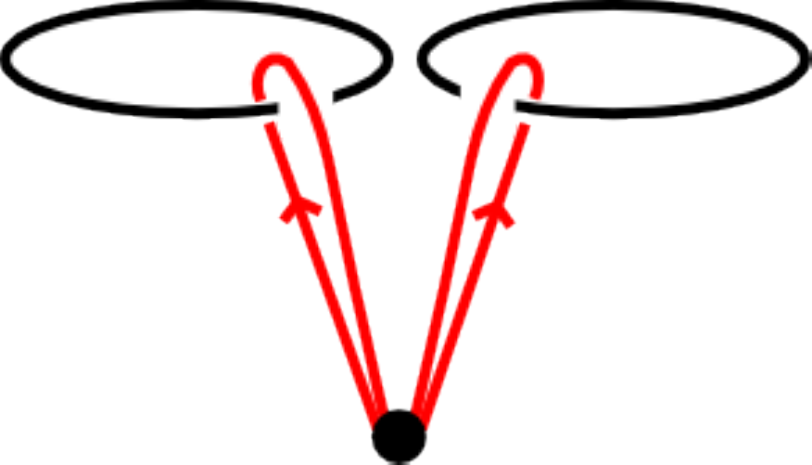}
\includegraphics[scale=.45]{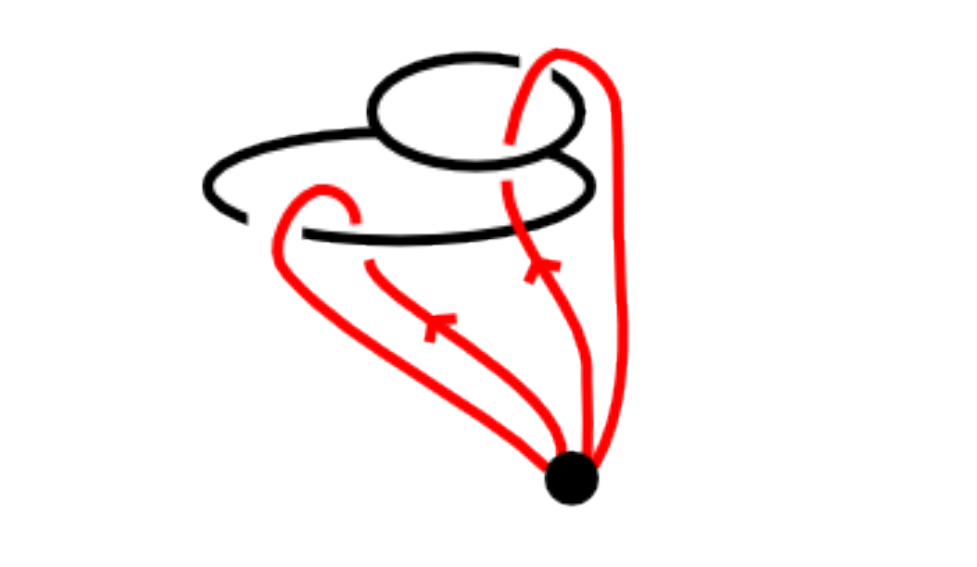}
\includegraphics[scale=.45]{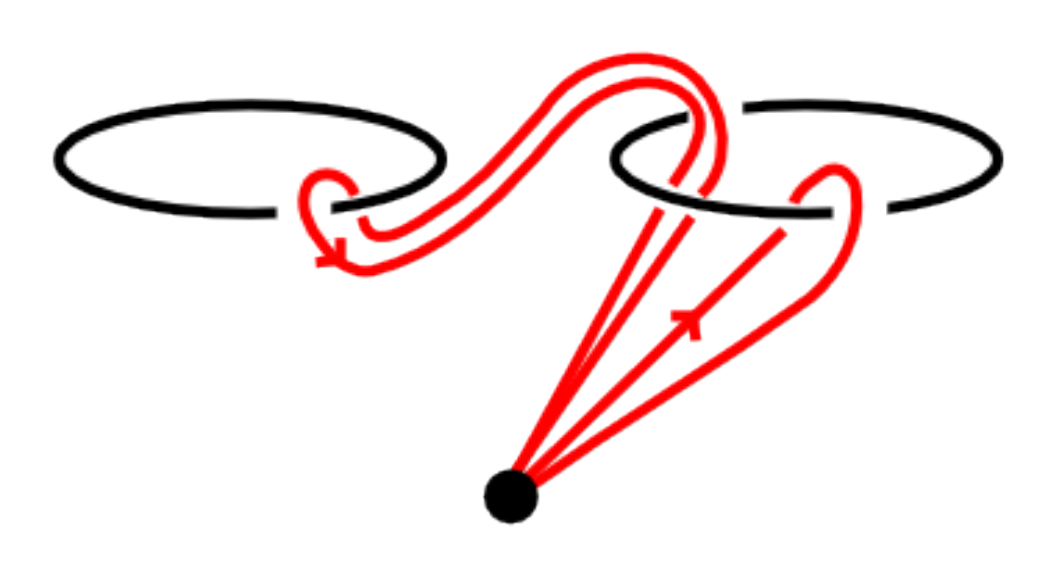}
      \caption{A 3-frame movie of the motion $\rho_1$ and the result on $\pi_1(M\bs\coprod_nS^k)$.}
      \label{figure:motion}
\end{figure}

Next we determine how $B_n\le\pi_1(\ca C)$ acts on $\pi_1\big(M\bs\coprod_nS^k\big)$. The homomorphism $\ca P$ of Proposition \ref{proposition:ringpush} gives a homomorphism $\pi_1(\ca C)\ra\pi_0\big(\ca D(M,[\phi])\big)$. The latter group acts on $\pi_1(M\bs \im\phi)$. If $M=\R^{k+2}$, then we can define $\pi_0\big(\ca D(M,[\phi])\big)\ra\text{Aut}\big(\pi_1(M\bs\im\phi, *)\big)$ act by identifying $\R^{k+2}\simeq\text{int}(\disk^{k+2})$, identifying $\ca D(M,[\phi])$ with the corresponding subgroup of $\Diff(\disk^{k+2})$, and choosing $*\in\partial\disk$. If $M=S^{k+2}$, then we cannot choose a global fixed point for the action of $\ca D(M,\phi)$ on $M\bs\im\phi$, and so we only obtain $\pi_0\big(\ca D(M,[\phi])\big)\ra\text{Out}\big(\pi_1(M\bs \im\phi)\big)$. 

The group $\pi_1(M\bs\im\phi)$ is free by the following lemma.

\begin{lemma}\label{lemma:fundamentalgroup}
Fix $k\ge2$. Let $\coprod_nS^k\hookrightarrow\R^{k+2}$ be an unlinked embedding. Then $\pi_1\big(\R^{k+2}\backslash\coprod_nS^k\big)$ is isomorphic to the free group $F_n$.
\end{lemma}

\begin{proof} For definiteness, we will specify a particular embedding $\phi: \coprod_nS^k \to \R^{k+2}$ where the $i^{th}$ sphere is mapped to the equator of the sphere of radius $1/4$ centered at $(i,0, \dots, 0) \in \R^{k+2}$. We proceed by induction on $n$. For the base case $n=1$, first note that 
\[S^{k+2}\cong\pa(\bb D^{k+1}\ti\bb D^2)=S^{k}\ti\bb D^2\bigcup_{S^{k}\ti S^1}\bb D^{k+1}\ti S^1.\]
It follows that $\pi_1\big(S^{k+2}\bs S^k)\cong\pi_1(\disk^{k+1}\ti S^1)\cong\Z$. Then also $\pi_1\big(\re^{k+2}\bs S^k\big)\cong\Z$, since removing a single point from a $(m\ge3)$-manifold does not change the fundamental group. 

For the inductive step, take $\phi$ as above and decompose $\re^{k+2}$ into open sets 
\[U=\{(x_1,\ld,x_{k+2}): x_1<n-\frac{1}{2}+\epsilon\}\>\>\text{ and }\>\>V=\{(x_1,\ld,x_{k+2}): x_1>n-\frac{1}{2}-\epsilon\}\] for any small positive $\epsilon$. By construction $U$ contains the first $n-1$ spheres and $V$ contains the $n^{th}$ sphere. Since $U\cap V$ is contractible, by Seifert--van Kampen, we have
\[\pi_1\big(\re^{k+2}\bs\coprod_nS^k\big)\cong\pi_1\big(\re^{k+2}\bs\coprod_{n-1}S^k\big)*\pi_1(\re^{k+2}\bs S^k)\cong F_{n-1}*\Z\cong F_n.\]
The second isomorphism uses the inductive hypothesis and the base case. 
\end{proof}

\begin{remark}The lemma obviously implies that $\pi_1\big(S^{k+2}\backslash\coprod_nS^k\big)\simeq F_n$. \end{remark}

We now have homomorphisms 
\[\beta: B_n\ra\pi_1(\ca C)\ra\pi_1\big(\Conf_n(S^k,\R^{k+2})\big)\ra\Aut(F_n)\]
and
\[\gamma: B_n\ra\pi_1(\ca C)\ra\pi_1\big(\Conf_n(S^k,S^{k+2})\big)\ra\Out(F_n).\]

To prove Proposition \ref{proposition:hasbraids} we show that $\beta$ is injective and that $\ker\gamma=Z(B_n)$. 

\begin{lemma}\label{lemma:injective}
The homomorphism $\beta$ is injective. 
\end{lemma}
\begin{proof}
There is another homomorphism $\alpha:B_n\ra\Aut(F_n)$ (sometimes called the \emph{Artin representation}) induced by the action of the mapping class group $\Mod(\bb D,n)\cong B_n$ on $\pi_1(\disk\bs\{n$ points$\})\cong F_n$. It is a well-known theorem of Artin that $\alpha$ is injective (see \cite{artin} or \cite[Corollary 1.8.3]{birman}). We prove the lemma by showing that $\beta$ and $\alpha$ coincide after making the right identifications. 

Choose a configuration $Y=\{y_1,\ld,y_n\}\sbs\disk$ as in Figure \ref{figure:disk}. Let $\{\sigma_1,\ld,\sigma_{n-1}\}$ be the standard generating set for $B_n$ (c.f.\ Lemma \ref{lemma:nonabelian}). The isomorphism $B_n\xrightarrow\sim\Mod(\disk,n)$ is defined by sending $\si_i$ to the mapping class that exchanges $y_i$ and $y_{i+1}$ by moving them counterclockwise around their midpoint. We choose generators $\eta_i$ for $\pi_1(\disk\bs Y,*)\cong F_n$ as in Figure \ref{figure:disk}. It is easy to compute (c.f.\ Figure \ref{figure:disk}) \[\alpha(\si_i):\left\{\begin{array}{rlll}\eta_j&\mapsto& \eta_j&\text{ if }j\neq i,i+1\\\eta_i&\mapsto& \eta_{i+1}\\\eta_{i+1}&\mapsto& \eta_{i+1}\eta_i\eta_{i+1}^{-1}\end{array}\right.\]

On the other hand, the inclusion $B_n\hookrightarrow\pi_1\big(\text{RConf}_n(S^k,\re^{k+2})\big)$ sends $\sigma_i$ to the motion $\rho_i$ defined above (Figure \ref{figure:motion}). We identify $\pi_1(\re^{k+2}\bs\coprod_nS^k)\cong F_n$ as follows. Fix a basepoint $*\in\re^{k+2}\bs\coprod_nS^k$, and choose an embedding $\coprod\disk^{k+1}\ra\re^{k+2}$ such that the boundary of the $i^{th}$ disk $D_i$ is the $i^{th}$ sphere. Then $\pi_1(\re^{k+2}\bs\coprod_nS^k,*)$ is generated by loops $\ga_1,\ld,\ga_n:[0,1]\ra\re^{k+2}\bs\coprod_nS^k$ such that $\ga_i\cap D_j=\emptyset$ for $i\neq j$ and $\ga_i$ has a single, positive transverse intersection with $D_i$. 
Then for any $\ga\in\pi_1(\re^{k+2}\bs\coprod_nS^k,*)$, expressing $\rho_i(\ga)\in F_n$ as a word in $\ga_1,\ld,\ga_n$ reduces to computing the intersection of $\rho_i(\ga)$ with the disks $D_1,\ld,D_n$. From this it is easy to see $\rho_i$ sends $\ga_i$ to $\ga_{i+1}$, sends $\ga_{i+1}$ to $\ga_{i+1}\ga_i\ga_{i+1}^{-1}$, and fixes $\ga_j$ for $j\neq i,i+1$; see Figure \ref{figure:motion2}. Since $\rho_i=\beta(\sigma_i)$, this shows that $\beta$ and $\alpha$ agree, as desired.
\end{proof}

\begin{lemma}
The kernel of $\gamma$ is equal to the center $Z(B_n)$. 
\end{lemma}
\begin{proof}
By the proof of Lemma \ref{lemma:injective}, $\gamma$ is the composition of the Artin representation $\alpha:B_n\ra\text{Aut}(F_n)$ with the projection $\text{Aut}(F_n)\ra\text{Out}(F_n)$. Thus it suffices to understand this composition. 

To describe $B_n\xrightarrow{\alpha}\text{Aut}(F_n)\ra\text{Out}(F_n)$, we use a stronger version of the theorem of Artin mentioned in the proof of Lemma \ref{lemma:injective} that describes the image of $\alpha$ explicitly. Let $F_n$ be generated by $\eta_1,\ld,\eta_n$ as in Figure \ref{figure:disk}. Then $\phi\in\text{im}(\al)$ if and only if there exists $A_1,\ld,A_n\in F_n$ and $\tau\in S_n$ such that 
\begin{enumerate}[(i)]
\item $\phi(\eta_i)=A_i\>\eta_{\tau(i)}\>A_i^{-1}$ for $1\le i\le n$, and 
\item $\phi(\eta_1\cdots\eta_n)=\eta_1\cdots\eta_n$. 
\end{enumerate} 
See \cite{artin} or \cite[Theorem 1.9]{birman}. From this it quickly follows that $\phi\in\text{Inn}(F_n)\cap\text{im}(\al)$ if and only if $\phi$ is conjugation by $(\eta_1\cd\eta_n)^r$ for some $r\in\Z$ (we must have $A_1=A_2=\cd=A_n$ and $A_1$ must commute with $\eta_1\cd\eta_n$). Now the lemma follows by checking that $\al(\Delta)=\text{conj}(\eta_1\cd\eta_n)$, where $\Delta\in B_n$ is the full twist (which generates $Z(B_n)$). 
\end{proof}
This completes the proof of Proposition \ref{proposition:hasbraids}. 
\end{proof}
\begin{figure}
\labellist 
\small\hair 2pt 
\pinlabel $y_1$ at 215 630 
\pinlabel $y_2$ at 270 630 
\pinlabel $y_n$ at 350 630 
\pinlabel $\eta_1$ at 200 677 
\pinlabel $\eta_2$ at 268 690 
\pinlabel $\eta_n$ at 335 676 
\pinlabel $*$ at 268 570 
\pinlabel $\alpha(\si_1)(\eta_2)$ at 480 685
\pinlabel $\alpha(\si_1)(\eta_1)$ at 490 555
\endlabellist 
  \centering
\includegraphics[scale=.55]{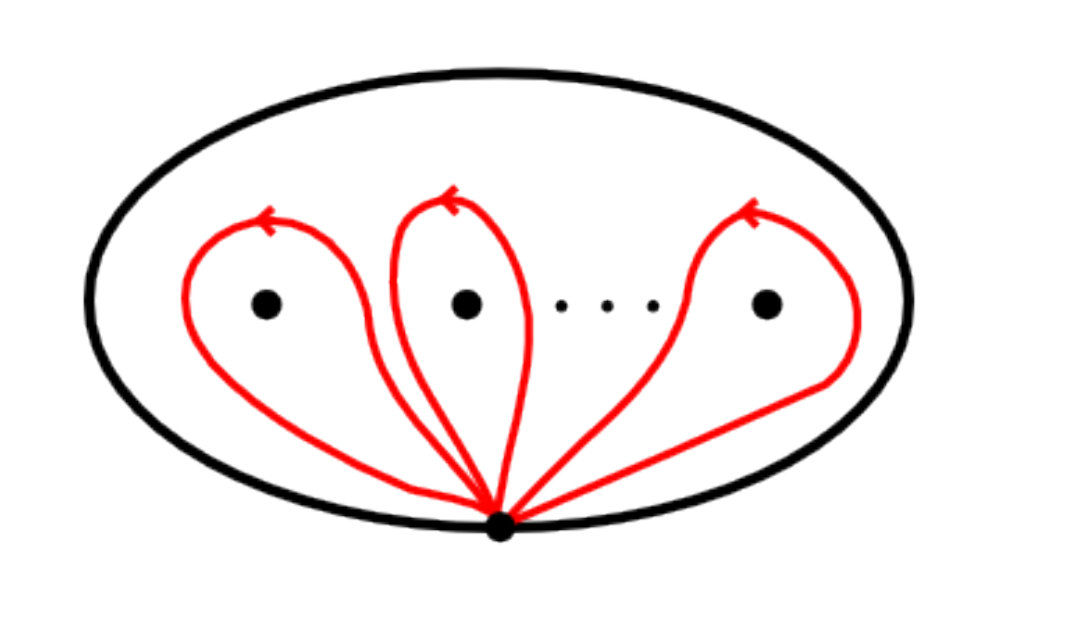}\hspace{.3in}
\includegraphics[scale=.55]{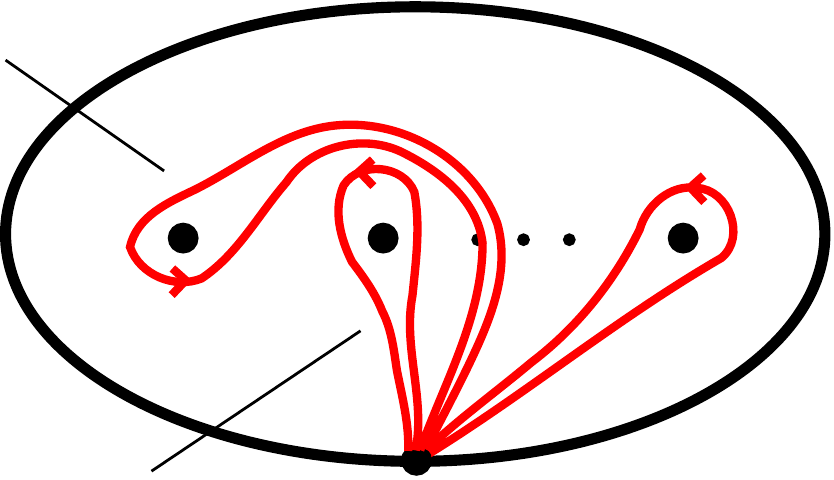}
      \caption{The braid group $B_n\cong\Mod(\bb D,n)$ acting on $\pi_1(\bb D\bs Y_n)$.}
      \label{figure:disk}
\end{figure}

\begin{figure}
\labellist 
\small\hair 2pt 
\pinlabel $\rho_1(\ga_1)$ at 405 405 
\pinlabel $\rho_1(\ga_2)$ at 200 405 
\endlabellist 
  \centering
\includegraphics[scale=.5]{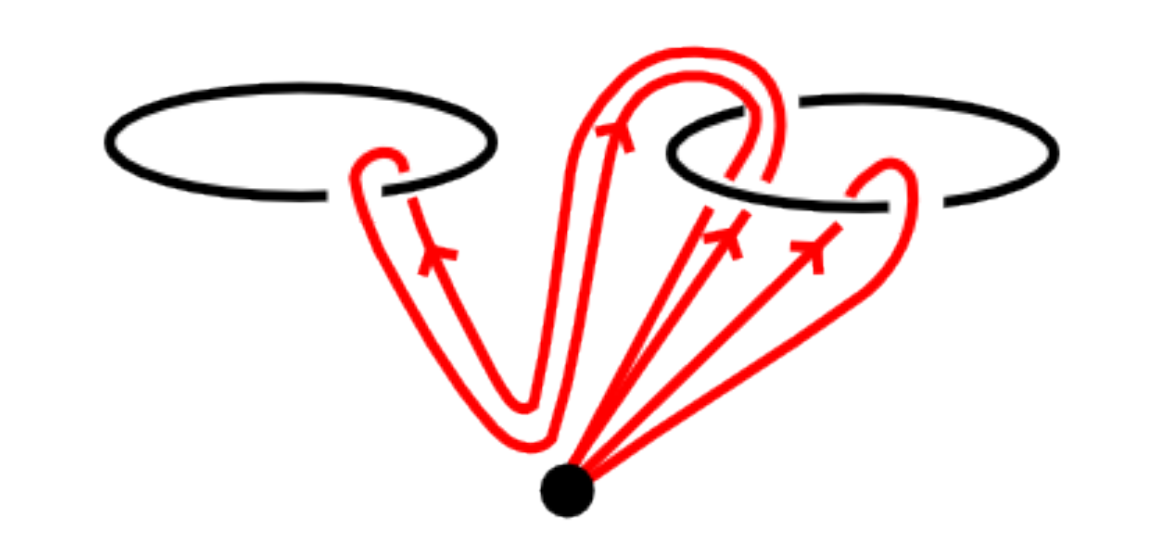}
      \caption{An illustration showing that $\rho_1(\ga_1)=\ga_2$ and $\rho_1(\ga_2)=\ga_2\ga_1\ga_2^{-1}$.}
      \label{figure:motion2}
\end{figure}

\para{Step 2: Reduction to Thurston stability} For $M=\R^{k+2}$ we will use the following easy corollary to Thurston stability (Theorem \ref{theorem:thurston_stability}).

\begin{corollary}\label{cor:ts}
Let $M$ be a noncompact manifold. Then the group $\Diff_c(M)$ of compactly supported $C^1$ diffeomorphisms is locally indicable (and hence any subgroup is also locally indicable).
\end{corollary}
\begin{proof}
Let $\Ga\le\Diff_c(M)$ be a finitely generated subgroup. The intersection of the supports of the generators is compact, so $\Ga$ acts trivially on a neighborhood of some $x\in M$. Thus $\Ga\le\ca G$, and there exists a surjection $\Ga\ra\Z$ by Thurston stability. 
\end{proof}

For the spherical motion groups $B_n(S^k,S^{k+2})$, there is one additional step that is required in the reduction process. Below, $\Diff(S^{k+2},S^k)$ denotes the group of diffeomorphisms of $S^{k+2}$ that restrict to the identity on the image of a fixed embedding $S^k\rightarrow S^{k+2}$. 
\begin{proposition}\label{proposition:B}
Let $\Gamma \le \Diff(S^{k+2},S^k)$ be finitely generated. If $\Gamma$ is strongly non-indicable, then there is a homomorphism $f: \Gamma \to GL^+_2(\R)$ with nonabelian image.
\end{proposition}

\begin{proof}
Choose $x \in S^k$. Then there are coordinates in which any $g \in \Diff(S^{k+2}, S^k)$ has derivative given by
\[
(Dg)_x = 
\left ( \begin{array}{c|c}
I_{k-2}	& V_g\\ \hline
0		& A_g
\end{array} \right).
\]
In this setting, $V_g \in M_{k-2,2}(\R)$ is a $(k-2) \times 2$ matrix, and $A_g \in GL^+_2(\R)$. Denote by $p: \Diff(S^{k+2},S^k) \to GL^+_2(\R)$ the homomorphism given by $p(g) = A_g$. 

Let $\Gamma \le \Diff(S^{k+2},S^k)$ be strongly non-indicable, and let $\Gamma' \le \Gamma$ be a finitely-generated perfect subgroup. We claim that $p: \Gamma \to GL^+_2(\R)$ has nonabelian image. If not, then $\Gamma' \le \ker p$. In this case, there is a map $V: \Gamma' \to M_{k-2, 2}(\R)$ defined by $V(g) = V_g$. As $M_{k-2, 2}(\R)$ is abelian and $\Gamma'$ is perfect, $V$ must be trivial. But then Thurston stability implies that $\Gamma'$ is locally indicable, a contradiction. 
\end{proof}

To complete the proof of Theorem \ref{theorem:spheres}, suppose $\sigma: B_n(S^k,M) \to\ca D(M, [\phi])$ is a lift of $\mathcal P$. If $M=\R^{k+2}$, then $B_n\le B_n(S^k,M)$ by Proposition \ref{proposition:hasbraids}. By Proposition \ref{proposition:ringpush}, $\si\big([B_n,B_n]\big)\le\ca D(M,[\phi])$ is a nontrivial subgroup. Since it is finitely-generated and perfect, $\ca D(M,[\phi])$ is strongly non-indicable. However, $\ca D(M,[\phi])\le\Diff_c(\R^k)$, so this contradicts Corollary \ref{cor:ts}.

In the case $M=S^{k+2}$, consider the homomorphism $j:B_n \ra B_n(S^k,M)$ provided by Proposition \ref{proposition:hasbraids}. Take a further subgroup $B_{n-1} \le B_n$ so that $\sigma\big(j(B_{n-1})\big)$ fixes some component of $\text{Im}(\phi)$ pointwise. By Propositions \ref{proposition:ringpush} and \ref{proposition:hasbraids}, the image of $B_{n-1}$ in $\ca D(S^{k+2}, [\phi])$ is nontrivial, and $\sigma([B_{n-1}, B_{n-1}])$ is a nontrivial finitely-generated perfect subgroup. Consequently $\sigma(B_{n-1})$ is strongly non-indicable. By Proposition \ref{proposition:B}, there is a homomorphism $f: \sigma(B_{n-1}) \to \GL_2^+(\R)$ with nonabelian image, but this contradicts Lemma \ref{lemma:nonabelian}. \qed

\section{Extensions of the main theorems}\label{section:extension}

In this section we give a strengthening of Theorems \ref{theorem:surfaces} and \ref{theorem:spheres} using a result of Parwani \cite[Theorem 1.4]{parwani} building off of work of Deroin-Kleptsyn-Navas \cite{dkn}. 

\begin{theorem}[Parwani]\label{theorem:parwani}
Let $G$ and $H$ be two finitely generated groups such that $H_1(G;\Z)=0=H_1(H;\Z)$. Then for any $C^1$ action of $G\ti H$ on $S^1$, either $G\ti1$ or $1\ti H$ acts trivially.
\end{theorem}

\subsection{Surfaces} Let $S$ be a closed surface and let $X\sbs S$ be finite. Let $S'$ be the compact surface obtained by replacing each marked point $x\in X$ with a boundary component. In what follows, $\Diff(S')$ denotes the group of diffeomorphisms of $S'$ where the boundary components of $S'$ are \emph{not} required to be fixed pointwise. It is well-known that $\pi_0\Diff(S,X)\cong\pi_0\Diff(S')$. Therefore, one can ask whether the homomorphism
\begin{equation}\label{eqn:lift1}\ca P:B_n(S)\ra\pi_0\Diff(S,X)\cong\pi_0\Diff(S')\end{equation}
lifts to a homomorphism $B_n(S)\ra\Diff(S')$. 

\begin{theorem}
Fix $n\ge11$. Then $\ca P:B_n(S)\ra\pi_0\Diff(S')$ is not realized by diffeomorphisms. That is, there does not exist a homomorphism $B_n(S)\ra\Diff(S')$ such that the composition $B_n(S)\ra\Diff(S')\ra\pi_0\Diff(S')$ is equal to $\ca P$. 
\end{theorem}
\begin{proof}
Suppose for a contradiction that $\si:B_n(S)\ra\Diff(S')$ is a lift of (\ref{eqn:lift1}). By passing to a finite-index subgroup of $B_n(S)$ we may assume one component $C\subset\pa S'$ is fixed. By the assumption $n\ge11$, this finite-index subgroup contains $B_5\ti B_5$. We may therefore take $G=[B_5,B_5]\times1$ and $H=1\times[B_5,B_5]$ in Theorem \ref{theorem:parwani} to conclude that, without loss of generality, $G$ acts trivially on $C$. As $G$ is non-locally-indicable (Theorem \ref{theorem:gorinlin}), the last stage of the argument of Theorem \ref{theorem:spheres} for the case $M = S^{k+2}$ can be applied to derive a contradiction.

\end{proof}

\subsection{Spheres}

Let $\Emb_n(S^k,\re^{k+2};\phi)$ be the embedding space defined in Section \ref{section:spherepush}. Define the \emph{(unparameterized) configuration space} $\overline{\Conf}_n(S^k,S^{k+2})$ as
\[\overline{\Conf}_n(S^k,S^{k+2})=\Emb_n(S^k,S^{k+2};\phi)/\Diff(\coprod_nS^k).\]
Note that $\overline{\Conf}_n(S^k,S^{k+2})$ is a quotient of $\Conf_n(S^k,S^{k+2})$, since $\Diff(\coprod_nS^k)$ is isomorphic to the wreath product $\Diff(S^k)\wr S_n$. An element of $\overline{\Conf}_n(S^k,S^{k+2})$ is a collection of disjoint, unordered, unlinked spheres (with no additional information about the parameterization). 

Fix $X\in\overline{\Conf}_n(S^k,S^{k+2})$, and let $\bar B_n(S^k,S^{k+2})=\pi_1\big(\overline{\Conf}_n(S^k,S^{k+2}),X\big)$. In the case $k=1$, this group coincides with the {\em ring group} studied by Brendle and Hatcher in \cite{brendle_hatcher}. By the argument in Proposition \ref{proposition:ringpush}, there is a homomorphism
\[\ca P:\bar B_n(S^k,S^{k+2})\ra\pi_0\big(\ca D(S^{k+2},X)\big),\]
where $\ca D(S^{k+2},X)\le\ca D(S^{k+2})$ is the subgroup of diffeomorphisms that preserve $X$ as a set. 

We have the following strengthening of Theorem \ref{theorem:spheres} in the case $k=1$. 
\begin{theorem}
Fix $n\ge15$. Then the homomorphism $\ca P:\bar B_n(S^1,S^3)\ra\pi_0\big(\ca D(S^{3},X)\big)$ is not realized by diffeomorphisms. 
\end{theorem}

\begin{proof}
Suppose for a contradiction that $\sigma: \bar B_n(S^1,S^3) \to \ca D(S^3, X)$ is a lift of $\mathcal P$. By the same argument as Proposition \ref{proposition:hasbraids}, there is a homomorphism $B_n\ra\bar B_n(S^1,S^3)$ with kernel contained in $Z(B_n)$. By passing to a finite-index subgroup of $\bar B_n(S^1,S^3)$, we may assume that one component $C \cong S^1 \subset X$ is fixed. By the assumption $n \ge 15$, this finite-index subgroup contains $B_7 \times B_7$ and {\em a fortiori} contains $[B_7, B_7] \times [B_7, B_7]$. Taking $G = [B_7,B_7] \times 1$ and $H= 1 \times [B_7, B_7]$ in Theorem \ref{theorem:parwani}, it follows that (without loss of generality) $G$ fixes $C$ pointwise.

For the remainder of the argument, we follow the strategy in Step 2 of Theorem \ref{theorem:spheres}. In order to be able to derive a contradiction from Proposition \ref{proposition:B}, we must have that every homomorphism $f: [B_7, B_7] \to \GL_2^+(\R)$ has abelian image.

The generating set $S$ of Proposition \ref{theorem:gorinlin}.(ii) satisfies the hypotheses of Lemma \ref{lemma:general} for $k = 3$. It follows that every homomorphism $f: [B_7, B_7] \to \GL_2^+(\R)$ has abelian image as desired. The argument in Step 2 of Theorem \ref{theorem:spheres} can now be carried out showing that $[B_7, B_7]\times 1 \le\bar B_n(S^1,S^3)$ lies in the kernel of any homomorphism $\sigma: \bar B_n(S^1,S^3) \to \ca D(S^3, X)$. Therefore $\bar B_n(S^1,S^3)$ cannot be realized by diffeomorphisms. 
\end{proof}

\bibliography{braids}{}

\begin{thebibliography}{DKN07}

\bibitem[Art25]{artin}
E.~Artin.
\newblock Theorie der {Z}\"opfe.
\newblock {\em Abh. Math. Sem. Univ. Hamburg}, 4(1):47--72, 1925.

\bibitem[BCS13]{bcs}
M.~Bestvina, T.~Church, and J.~Souto.
\newblock Some groups of mapping classes not realized by diffeomorphisms.
\newblock {\em Comment. Math. Helv.}, 88(1):205--220, 2013.

\bibitem[BH13]{brendle_hatcher}
T.~Brendle and A.~Hatcher.
\newblock Configuration spaces of rings and wickets.
\newblock {\em Comment. Math. Helv.}, 88(1):131--162, 2013.

\bibitem[Bir74]{birman}
J.~Birman.
\newblock {\em Braids, links, and mapping class groups}.
\newblock Princeton University Press, Princeton, N.J.; University of Tokyo
  Press, Tokyo, 1974.
\newblock Annals of Mathematics Studies, No. 82.

\bibitem[Dah62]{dahm}
D.~Dahm.
\newblock {\em A Generalization of Braid Theory}.
\newblock ProQuest LLC, Ann Arbor, MI, 1962.
\newblock Thesis (Ph.D.)--Princeton University.

\bibitem[DKN07]{dkn}
B.~Deroin, V.~Kleptsyn, and A.~Navas.
\newblock Sur la dynamique unidimensionnelle en r\'egularit\'e interm\'ediaire.
\newblock {\em Acta Math.}, 199(2):199--262, 2007.

\bibitem[EE69]{ee}
C.~Earle and J.~Eells.
\newblock A fibre bundle description of {T}eichm\"uller theory.
\newblock {\em J. Differential Geometry}, 3:19--43, 1969.

\bibitem[ES70]{es}
C.~J. Earle and A.~Schatz.
\newblock Teichm\"uller theory for surfaces with boundary.
\newblock {\em J. Differential Geometry}, 4:169--185, 1970.

\bibitem[FH09]{frankshandel}
J.~Franks and M.~Handel.
\newblock Global fixed points for centralizers and {M}orita's theorem.
\newblock {\em Geom. Topol.}, 13(1):87--98, 2009.

\bibitem[FM12]{fm}
B.~Farb and D.~Margalit.
\newblock {\em A primer on mapping class groups}, volume~49 of {\em Princeton
  Mathematical Series}.
\newblock Princeton University Press, Princeton, NJ, 2012.

\bibitem[GL69]{gorinlin}
E.~A. Gorin and V.~Ja. Lin.
\newblock Algebraic equations with continuous coefficients, and certain
  questions of the algebraic theory of braids.
\newblock {\em Mat. Sb. (N.S.)}, 78 (120):579--610, 1969.

\bibitem[LS77]{lyndonschupp}
R.~Lyndon and P.~Schupp.
\newblock {\em Combinatorial group theory}.
\newblock Springer-Verlag, Berlin-New York, 1977.
\newblock Ergebnisse der Mathematik und ihrer Grenzgebiete, Band 89.

\bibitem[Mar07]{markovic}
V.~Markovic.
\newblock Realization of the mapping class group by homeomorphisms.
\newblock {\em Invent. Math.}, 168(3):523--566, 2007.

\bibitem[Mil68]{milnorbook}
J.~Milnor.
\newblock {\em Singular points of complex hypersurfaces}.
\newblock Annals of Mathematics Studies, No. 61. Princeton University Press,
  Princeton, N.J.; University of Tokyo Press, Tokyo, 1968.

\bibitem[Mor87]{morita_nonlifting}
S.~Morita.
\newblock Characteristic classes of surface bundles.
\newblock {\em Invent. Math.}, 90(3):551--577, 1987.

\bibitem[M{\v{S}}08]{markovicsaric}
V.~Markovic and D.~{\v{S}}ari{\'c}.
\newblock The mapping class group cannot be realized by homeomorphisms.
\newblock http://arxiv.org/pdf/0807.0182v1.pdf, 2008.

\bibitem[Nar15]{nariman}
S.~Nariman.
\newblock Braid groups and discrete diffeomorphisms of the punctured disk.
\newblock in progress, Sept. 2015.

\bibitem[Pal60]{palais}
R.~Palais.
\newblock Local triviality of the restriction map for embeddings.
\newblock {\em Comment. Math. Helv.}, 34:305--312, 1960.

\bibitem[Par08]{parwani}
K.~Parwani.
\newblock {$C^1$} actions on the mapping class groups on the circle.
\newblock {\em Algebr. Geom. Topol.}, 8(2):935--944, 2008.

\bibitem[PR00]{parisrolfsen}
L.~Paris and D.~Rolfsen.
\newblock Geometric subgroups of mapping class groups.
\newblock {\em J. Reine Angew. Math.}, 521:47--83, 2000.

\bibitem[Thu74]{thurston_stability}
W.~Thurston.
\newblock A generalization of the {R}eeb stability theorem.
\newblock {\em Topology}, 13:347--352, 1974.

\bibitem[Thu11]{thurston_mo}
W.~Thurston.
\newblock Realizing the braid group by homeomorphisms.
\newblock
  \url{http://mathoverflow.net/questions/55555/realizing-braid-group-by-homeomorphisms},
  February 2011.

\end{thebibliography}
\bibliographystyle{alpha}

\end{document}